\documentclass[a4paper,11pt]{article}
\usepackage{stmaryrd}
\usepackage[utf8]{inputenc} 
\usepackage{amsfonts}
\usepackage{amssymb}
\usepackage{amsthm}
\usepackage{amsmath}
\usepackage{a4wide}
\usepackage{mathrsfs}
\usepackage{epsfig}
\usepackage{comment}
\usepackage{palatino}
\usepackage{tikz}
\usepackage{fp,ifthen}
\usepackage{esint}
\usepackage{nicefrac}
\usetikzlibrary{decorations.pathreplacing}
\usepackage{float}
\usepackage{enumerate} 
\usepackage{enumitem} 

\usepackage[a4paper,twoside,top=2.5cm, bottom=2cm, left=2.3cm, right=2.3cm]{geometry} 

 \usepackage{lipsum}

\usepackage[colorinlistoftodos,textwidth=2.3cm]{todonotes} 

\newcommand{\pdfgraphics}{\ifpdf\DeclareGraphicsExtensions{.pdf,.jpg}\else\fi}
\usepackage{graphicx}

\usepackage{color}
\definecolor{hanblue}{rgb}{0.27, 0.42, 0.81}
\definecolor{red}{rgb}{1.0, 0.0, 0.0}
\usepackage[colorlinks, citecolor=blue,linkcolor=blue, urlcolor = blue]{hyperref}
\usepackage[english,capitalize]{cleveref}

\usepackage{mathtools}

\theoremstyle{plain}

\newtheorem{theorem}{Theorem}[section]
\newtheorem{lemma}[theorem]{Lemma}
\newtheorem{proposition}[theorem]{Proposition}

\theoremstyle{definition}
\newtheorem{definition}[theorem]{Definition}

\newcommand{\dist}{\operatorname{d}}

\theoremstyle{remark}

\newcommand{\m}{\ensuremath{\mathfrak m}}

\DeclareMathOperator\Lip{Lip}
\newcommand{\X}{\mathsf{X}} 
\newcommand{\Y}{\mathsf{Y}} 

\numberwithin{equation}{section}

\newcommand{\de}{\ensuremath{\,\mathrm d}} 
\renewcommand{\d}{\ensuremath{\mathrm d}} 

\renewcommand{\epsilon}{\varepsilon}
\newcommand{\N}{\ensuremath{\mathbb N}}
\newcommand{\R}{\ensuremath{\mathbb R}}

\DeclareMathOperator*{\argmin}{arg\,min}

\makeatletter
\renewcommand*\env@matrix[1][*\c@MaxMatrixCols c]{%
  \hskip -\arraycolsep
  \let\@ifnextchar\new@ifnextchar
  \array{#1}}
\makeatother

\pdfgraphics 
\begin{document}
\title{Lipschitz continuity of harmonic maps between ${\rm RCD}(K,N)$ spaces and ${\rm CAT}(\kappa)$ spaces}

\author{Luca Gennaioli \footnote{\href{L.Gennaioli@warwick.ac.uk}{L.Gennaioli@warwick.ac.uk}}}

\date{\today}

\maketitle

\begin{abstract}
We are going to prove that energy minimizing harmonic maps from a domain $\Omega$ inside an ${\rm RCD}(K,N)$ whose image lies in a small ball inside a ${\rm CAT}(\kappa)$ space are locally Lipschitz continuous. This completes the picture concerning regularity of harmonic maps between singular spaces and justifies a variant of the Bochner-Eells-Sampson inequality between singular spaces.
\end{abstract}
\tableofcontents

\section{Introduction}
This paper is concerned with the regularity property of harmonic maps from ${\rm RCD}(K,N)$ spaces to ${\rm CAT}(\kappa)$ spaces. The former consists of all ${\rm CD}(K,N)$ metric measure spaces where, roughly speaking, one can impose a lower bound on the Ricci curvature $K$, an upper bound on the dimension $N$ and impose an hilbertian structure. Indeed general ${\rm CD}(K,N)$ in general do not enjoy several "Riemannian" theorems, like the splitting theorem for example. This Riemannian structure was added in \cite{Gigli12}, where the author (see \cite{gigli2023giorgi} for a rather complete account on this subject) introduced the notion of \emph{infinitesimal hilbertianity}, singling out the class of ${\rm RCD}(K,N)$ spaces. Of course the various contributions to this topic are numerous. We recall here some of the most significant ones: \cite{AmbrosioGigliSavare11}, \cite{AmbrosioGigliSavare11-2} (the infinite dimensional case), \cite{AmbrosioGigliSavare12}, \cite{AmbrosioGigliSavare-compact}, \cite{AmbrosioGigliMondinoRajala12}, \cite{Gigli14},  \cite{Gigli13} and \cite{Gigli-Kuwada-Ohta10}.
\newline
For what concerns ${\rm CAT}(\kappa)$ spaces, they are metric spaces where one can express synthetically the fact of having sectional curvature bounded from above by a value $\kappa$ and such terminology was used for the first time in \cite{G87}.
\newline
To convince the reader of the fact that this is the natural non-smooth setting where one can hope to obtain "good" regularity properties of harmonic functions (arguably the class of functions which should be the most regular), we recall that if $u:\Omega\subset (M^n,g)\to (N^m,h)$ with ${\rm Ric}_{g_M}\geq0$ (to illustrate the principle we stick to nonnegative Ricci curvature but a general lower bound would be enough) and ${\rm Sec}_{g_N}\leq0$ (again to illustrate the principle we stick to nonpositive sectional curvature) is harmonic, one can write the Bochner-Eells-Sampson formula 
\begin{equation}
\label{eq:BES_formula}
    \Delta\bigg(\frac{|\de u|_{\sf HS}^2}{2}\bigg)=|\nabla\de u|_{\sf HS}^2+{\rm Ric}_{g_M}(\nabla u,\nabla u)-\sum_{i,j\leq n}\langle u_*\mathcal{R}^N(e_i,e_j)e_i,e_j\rangle,
\end{equation}
where $\mathcal{R}^N$ is the Riemann tensor of $N$ and $|\nabla\de u|^2_{\rm HS}$ denotes the square Hilbert-Schmidt norm of the Hessian of $u$. Using the inequalities on the Ricci and on the sectional curvature we get
\begin{equation*}
    \Delta\bigg(\frac{|\de u|_{\sf HS}^2}{2}\bigg)\geq|\nabla\de u|_{\sf HS}^2\geq 0,
\end{equation*}
and a simple application of Harnack's inequality allows to infer that $|\de u|_{\rm HS}$ is locally bounded, proving local Lipschitz regularity of $u$. 
\newline
For what concerns the non-smooth setting, the problem of regularity is tied to the availability of \eqref{eq:BES_formula}, which was instead a tool to exploit in the smooth setting. The first contributions appeared for what concerns local Lipschitz regularity of harmonic maps between Alexandrov spaces (i.e. with bounds on the sectional curvature): for real-valued maps we have \cite{Pet03}, while the problem was solved in \cite{ZZ18}. For what concerns the case of a smooth domain inside a manifold $(M,g_M)$ with ${\rm Ric}_{g_M}\geq K$ and a non-smooth target with sectional curvature bounded above by $\kappa$, in \cite{Ser95} was able to achieve the sought regularity because, strongly exploiting the smoothness, he was able to write a PDE of the form
\begin{equation}
    \label{eq:Serb_PDE}
    \tfrac{|\de u|_{\sf HS}}{\cos(f)}{\rm div}\Big(\cos^2(f)\nabla\Big(\tfrac{|\de u|_{\sf HS}}{\cos(f)}\Big)\Big)\geq K|\de u|^2_{\sf HS},
\end{equation}
where $f(x):=\sqrt{\kappa}{\sf d}_\Y(u(x),p)$ and $(\Y,\d_\Y)$ is the target ${\rm CAT}(\kappa)$ space and $u:\Omega\subset M\to\Y$ is a minimizing harmonic map with $u(\Omega)\subset B_\rho(o)$, for some $o\in\Y$ and $\rho<\frac{\pi}{2\sqrt{\kappa}}$. Equation \eqref{eq:Serb_PDE} allows for a Moser iteration and again implies the local Lipschitz regularity. We want to remark that the constraint on the image of $u$, namely that $u(\Omega)\subset B_\rho(o)$, is not a technicality but it is crucial to obtain such regularity and even continuity. The heuristic reason for which this is important is that the map $y\mapsto\d_\Y(o,y)$ cease to be geodesically convex outside $B_\rho(o)$. Indeed if one considers the vortex map $u:B_1(0)\subset\R^7\to\mathbb S^6$ with $u(x)=\frac{x}{|x|}$, such map is an energy minimizing harmonic map (among maps with the same boundary datum) and it is clearly not continuous (see \cite{JK83}, \cite{L87} and \cite{CG89}).
Finally, moving to the ${\rm RCD}$ setting for what concerns the domain (which in view of the previous discussion appears to be the natural setting), in the works \cite{Gig23} and \cite{MS22} the authors were able to independently establish the local Lipschitz regularity of energy minimizing harmonic maps, together with a variant of the Bochner-Eells-Sampson formula where one has to replace $|\de u|_{\rm HS}$ with ${\rm lip}(u)$ , i.e. the local Lipschitz constant of $u$ (see \cite[Theorem 7.1]{MS22} for the inequality with an Hessian-type term). However both of the previous work do not pursue the regularity of maps into general ${\rm CAT}(\kappa)$ spaces as their strategy is designed to exploit a variant of \eqref{eq:BES_formula} in a crucial way. When $\kappa>0$ such an inequality cannot immediately be exploited as on the r.h.s. one would have a term of the type $\kappa\|\de u\|_{\rm HS}^4$, which is not clearly an $L^1_{\rm loc}$ function. Moreover, using their proof strategy, to actually be even able to write this term one would already need to require that the map $u$ is locally Lipschitz. Therefore one needs to completely change the strategy and step by step try to upgrade the regularity of the map $u$ and more precisely of its Hopf-Lax regularization. Using this type of regularization was the main idea introduced in \cite{ZZ18}, which led the authors in \cite{Gig23} and \cite{MS22} to adapt this strategy to their setting (in an extremely non-trivial way).
\newline
In this paper we are going to exploit ourselves this Hopf-Lax regularization scheme and, building on the recent higher integrability of $|\de u|_{\rm HS}$ for harmonic maps (see Theorem \ref{thm:higher_integrability}) established in \cite{GGZZ26}. Another crucial tool will be the H\"older regularity of the map $u$ (see \ref{thm:Holder_reg}, established in \cite[Theorem 3.12]{GGZZ26}). Finally, exploting all of these results and a further regularization of the Hopf-Lax semigroup, we are going to establish an inequality of the type \eqref{eq:Serb_PDE} (see \eqref{eq:Serb_ineq}). The latter will then yield the sought local Lipschitz continuity via a standard argument.
\newline
We shall state and prove all of our result in the setting of ${\rm CAT}(1)$ spaces, as the case of ${\rm CAT}(\kappa)$ spaces with $\kappa>0$ can be recovered by scaling the distance function on the target space. Our main Theorem is the following.
\begin{theorem}
\label{thm:Lipschitz_continuity}
Let $(\X,\d_\X,\m)$ be an $\operatorname{RCD}(K,N)$ space, and let
$(\Y,\d_\Y)$ be a $\operatorname{CAT}(1)$ space. Let
\begin{equation*}
u\colon \Omega\subset X\longrightarrow Y
\end{equation*}
be a harmonic map with $\Omega$ bounded open set and $\m(\X\setminus\Omega)>0$ such that
\begin{equation*}
u(\Omega)\subset B_\rho(o),
\qquad
\rho<\frac{\pi}{2}.
\end{equation*}
Then $u$ is locally Lipschitz continuous.
\end{theorem}
Once this Theorem is established one can finally and rigorously write a variant of the Bochner-Eells-Sampson formula.
\begin{theorem}
\label{thm:BES_ineq}
Let $(\X,\d_\X,\m)$ be an ${\rm RCD}(K,N)$ space, and let $(\Y,\d_\Y)$ be a ${\rm CAT}(1)$ space. Let
\begin{equation*}
u:\Omega\subset X\to Y
\end{equation*}
be a harmonic map with $\Omega$ open and $\m(\X\setminus\Omega)>0$ such that
\begin{equation*}
u(\Omega)\subset B_\rho(o),
\qquad
\rho<\frac{\pi}{2}.
\end{equation*}
Then the inequality  
\begin{equation}
    \label{eq:Bochner_inequality}
    \Delta\bigg(\frac{{\rm lip}^2u}{2}\bigg)\geq|\nabla{\rm lip}u|^2+K{\rm lip}^2(u)-\frac{|\de u|_{\rm HS}^2{\rm lip}^2u}{n+2}
\end{equation}
holds in the weak sense in $\Omega$, where $n$ is the essential dimension of $(\X,\d,\m)$.
\end{theorem}
We are not going to prove the latter Theorem as it was already proved in \cite[Theorem 3.26]{GGZZ26}, \emph{assuming} the local Lipschitz continuity of $u$, which was not available at the time. Finally we refer to \cite[Theorem 3.27]{GGZZ26} for the sharp Lipschitz bound on $u$ and to \cite[Corollary 3.28]{GGZZ26} for what concerns an $L^\infty$-Liouville result for $u:\X\to\Y$ which is globally defined on $\X$, which is an ${\rm RCD}(0,N)$ space, with $\Y$ which is a ${\rm CAT}(\kappa)$ space and $u(\X)\subset B_\rho(o)\subset\Y$, with $\rho<\frac{\pi}{2\sqrt{\kappa}}$: all the previous results are now rigorous, since the local Lipschitz continuity has been established.
\newline
The paper is organized as follows: in Section \ref{sec:preliminaries} we are going to introduce the main tools and notation for what concerns ${\rm RCD}$ spaces, ${\rm CAT}$ spaces and the notion of minimizing harmonic map, along with several useful and known results concerning the topic. In Section \ref{sec:main_res} instead we are going to prove state and prove the main results building on top of which we are going to establish Theorem \ref{thm:Lipschitz_continuity}.

\section{Preliminaries}
\label{sec:preliminaries}
In this section $(\X,\dist_\X,\m)$ will be denoting an ${\rm RCD}(K,N)$ space with $N<+\infty$ and $(\Y,\dist_\Y)$ will be a ${\rm CAT}(1)$ space.  We stress again that since the statements for ${\rm CAT}(\kappa)$ follow by scaling, we shall stick to the setting of ${\rm CAT}(1)$ here and in the following. Moreover we assume the reader to be familiar with the basic theory of ${\rm RCD}$ and ${\rm CAT}$ spaces, as we shall only recall the fundamental tools needed for our purposes.
We shall introduce the theory developed in \cite{GT20} (after the seminal work \cite{KS93}) and define the Korevaar-Schoen energy, in order to define minimizing harmonic maps.

Let $u\in L^2(\Omega,\Y)$ with $\Omega\subseteq\X$ open set. We call the $2$-energy density of $u$ at scale $r$ inside $\Omega$ the quantity $ \mathbf{ks}_{2,r}[u,\Omega]:\X\to\R_+$, defined as 
\begin{equation}
    \label{p-energy density}
    \mathbf{ks}_{2,r}[u,\Omega](x):=
    \begin{cases}
    \bigg(\fint_{B_r(x)}\frac{\dist_\Y^2(u(x),u(y))}{r^2}\de\m(y)\bigg)^{\frac{1}{2}}\quad&{\rm if}\;B_r(x)\subset\Omega\\
    0&{\rm otherwise}.
    \end{cases}
\end{equation}
Moreover we introduce the \emph{total energy} of $u$ in $\Omega$ as \begin{equation}
\label{KS energy}
    {\rm E}_2[u,\Omega]:=\liminf_{r\to 0}\int_{\Omega}\mathbf{ks}_{2,r}[u,\Omega]^2(x)\de\m(x).
\end{equation}
We can now define Sobolev spaces as follows
\begin{definition}[Korevaar-Schoen space and harmonic maps]
    We say that a function $u\in L^2(\Omega,\Y)$ is in ${\rm KS}^{1,2}(\Omega,\Y)$ if ${\rm E}_2[u,\Omega]<+\infty$. Moreover, given $w\in {\rm KS}^{1,2}(\Omega,\Y)$, we say that $u$ is \emph{harmonic} in $\Omega$ with boundary datum $w$, if $u\in\argmin_{v\in {\rm KS}^{1,2}_{w}(\Omega,\Y)} {\rm E}_{2}[v,\Omega]$, where
    \begin{equation*}
        {\rm KS}^{1,2}_{w}(\Omega,\Y):=\bigg\{v\in  {\rm KS}^{1,2}(\Omega,\Y):\quad \dist_\Y(v,w)\in W^{1,2}_0(\Omega)\bigg\}.
    \end{equation*}
\end{definition}
The existence theory for minimizers of ${\rm E}_2[\cdot,\Omega]$ has been carried out in \cite{Sak23} (see Theorem 1.2 therein) under the condition that the boundary datum has image contained in a sufficiently small ball of the target space.

The next result, which can be found in \cite[Theorem 3.13]{GT20}, provides a representation formula for the Korevaar-Schoen energy density.
\begin{theorem}
    Let $(\X,\dist_\X,\m)$ be an ${\rm RCD}(K,N)$ space and $(\Y,\dist_\Y)$ a complete metric space. Then for every $u\in{\rm KS}^{1,2}(\X,\Y)$ there exists a function $e_2[u]\in L^2(\X)$, called \emph{energy density} of $u$, such that 
    \begin{equation*}
        \mathbf{ks}_{2,r}[u]\to e_2[u]\quad\m-{\rm a.e.\;\;and\;in}\;L^2\;{\rm as}\;r\to 0.
    \end{equation*}
    In particular the $\liminf$ in \eqref{KS energy} is actually a limit.
\end{theorem}
The following is instead an an equivalent way to speak about $e_2[u]$ in terms, based on the Hilbert-Schmidt norm of the differential $|\de u|_{\rm HS}$: we are not going to discuss the meaning of the object $\de u$, referring to \cite{GPS18} for the details. What follows is \cite[Proposition 6.7]{GT20}.
\begin{theorem}
    Let $(\X,\dist_\X,\m)$ be an ${\rm RCD}(K,N)$ space and $\Omega\subset\X$ an open set. Let $(\Y,\dist_\Y)$ be a ${\rm CAT}(\kappa)$ space and $u\in{\rm KS}^{1,2}(\Omega,\Y)$, then for its energy density we have the following representation formula
    \begin{equation}
        \label{HS norm and energy density}
        e_2[u]=(d+2)^{-\frac{1}{2}}|\de u|_{\rm HS}.
    \end{equation}
\end{theorem}
\begin{definition}
Let $(\X,\dist_\X,\m)$ be a metric measure space and $\Omega\subset\X$ an open and bounded set. Let $\eta:\Omega\to\R$ be locally integrable. We say that a function $f\in W^{1,2}_{\rm loc}(\Omega)$ is such that $\Delta f\leq \eta$ weakly in $\Omega$ if for all $\varphi\in\Lip_c(\Omega)$ with $\varphi\geq 0$ we have 
    \begin{equation*}
        -\int_{\X}\nabla f\cdot\nabla\varphi\de\m_\X\leq\int_\X\varphi\eta\de\m_\X.
    \end{equation*}
\end{definition}
We now introduce some further notation trying to stick to the one of \cite{GGZZ26}. We let $U\Subset\Omega$ be open and consider $r>0$ and $\hat{x}\in U$ such that $B_{4r}(\hat{x})\subset U$ and $\|u\|_{C^\alpha(U)}(2r)^\alpha<s<\pi/8$. Finally call $B=B_r(\hat{x}),$ $2B:=B_{2r}(\hat{x})$, $\widehat{B}=B_{3r/2}(\hat{x})$ and $4B=B_{4r}(\hat{x})$. Let $o_B=u(\hat{x})$ we then have, thanks to the previous conditions, $u(2B)\subset B_s(o_B)$.

For $p>2$, define the function $F:\R\to\R$ as
\begin{equation*}
F(s)
:=
2\sin\left(\frac{s}{2}\right)
+
4\sin^2\left(\frac{s}{2}\right).
\end{equation*}
and $f:\X\times\X\to[-6,0]$ as
\begin{equation*}
    f(x,y)= 
    \begin{cases}
       -F(\dist_\Y(u(x),u(y)))\quad&{\rm if}\,x,y\in \widehat{B} \\
       -6\quad&{\rm otherwise}.
    \end{cases}
\end{equation*}
Then for all $t>0$ define the function $f_t:B\to\R$ as 
\begin{equation*}
f_t(x)
:=
\inf_{y\in\X}
\left[
\frac{\d_\X^p(y,x)}{p\,t^{p-1}}+f(x,y)
\right].
\end{equation*}
Now denoting with $y_{t,x}$ a minimizer for $f_t(x)$, and choosing $x$ as a competitor, we get
\begin{equation*}
    f_t(x)\leq \frac{\dist_\X^p(x,y_{t,x})}{pt^{p-1}}+f(x,y_{t,x})\leq 0.
\end{equation*}
This means $\dist_\X(x,y_{t,x})\leq 6p\frac{1}{p}t^\frac{p-1}{p}$. We now further shrink $t_*$ in a way that $(6p)^\frac{1}{p}t_*^{\frac{p-1}{p}}<{\rm dist}(\overline{B},\X\setminus \widehat{B})$ so that 
\begin{equation*}
      f_t(x):=\inf_{y\in \widehat{B}}\bigg[\frac{\dist_\X^p(x,y)}{pt^{p-1}}-F(\dist_\Y(u(x),u(y)))\bigg]\qquad\forall x\in B
\end{equation*}
Thanks to Theorem \ref{thm:Holder_reg} below we observe that $u$ is locally H\"older continuous inside $\Omega$ and thanks to \cite[Lemma 3.16]{GGZZ26} so is $f_t$.
Set
\begin{equation*}
\gamma_t(x):=-\frac{f_t(x)}{t},
\end{equation*}
\begin{equation*}
S_t(x)
:=
\left\{
y\in\X:
f_t(x)
=
\frac{\d_\X^p(y,x)}{p\,t^{p-1}}+
f(x,y)\right\},
\end{equation*}
and
\begin{equation*}
L_t(x)
:=
\min_{y\in S_t(x)}\d_\X(y,x),
\qquad
a_t(x)
:=
\frac{L_t^p(x)}{p\,t^p},
\qquad
\frac{D_t(x)}{t}
:=
\gamma_t(x)+a_t(x).
\end{equation*}
We now recall\cite[Theorem 3.15]{GGZZ26}, which proves higher integrability for gradients of minimizing harmonic maps.
\begin{theorem}
    \label{thm:higher_integrability}
    Let $\Omega, Y$ as above and $u:\Omega\to\Y$ be an harmonic map with $u(\Omega)\subset B_\rho(o)$ and $\rho<\pi?2$. Then there exists an $\varepsilon=\varepsilon(N,K,{\rm diam}(\Omega),\rho)>0$ such that $|\de u|_{\rm HS}\in L^{2+\epsilon}_{\rm loc}(\Omega)$ and  
    \begin{equation}\label{doubling-energy}
    \left(\fint_{B} |\de u|^{2+\varepsilon}_{\rm HS}\de\m\right)^{\frac{2}{2+\varepsilon}}\leqslant C_\varepsilon\fint_{2B}|\de u|_{\rm HS}^2\de\m
    \end{equation}
    for any ball $B$ with $2B\subset \Omega,$ where the constant $C_\epsilon>0$ depends only on $\epsilon$.
    \end{theorem}
    We now recall \cite[Theorem 3.12]{GGZZ26} which establishes the local H\"older regularity of $u$.
\begin{theorem}
\label{thm:Holder_reg}
    Let $(\X,\dist,\m)$ be an ${\rm RCD}(K,N)$ space and $(\Y,\d_\Y)$ be a ${\rm CAT}(1)$ space. Let $u:\Omega\subseteq\X\to\Y$ be a harmonic map such that ${\rm Im}(u)\subseteq B_{\rho}(o)$ with $\rho<\pi/2$. Then $u$ is locally H\"older continuous in $\Omega$.
\end{theorem}
For the reader convenience we present \cite[Lemma 3.19]{GGZZ26}, which is a crucial differential inequality that we are going to exploit later on.
\begin{lemma}
\label{lem:key_lemma}
Let $u:\Omega\to\Y$ be a harmonic map with $\Omega\subset\X$ open set, $(\Y,\dist_\Y)$ which is a ${\rm CAT}(1)$ space and ${\rm Im}(u)\subseteq B_\rho(o)$ with $o\in\Y$, $\rho<\pi/2$. Let further $f_o(x):=\cos(\dist_{\Y}(u(x),o))$, then we have $f_o\in W^{1,2}(\Omega)$ and 
\begin{equation}
    \label{eq:key_inequality}
    \Delta f_o\leq -f_o |\de u|_{\rm HS}^2=-f_o(n+2) e_2^2[u],
\end{equation}
    weakly in $\Omega$.
\end{lemma}

We also recall here \cite[Proposition 3.24]{GGZZ26}, which gives a crucial distributional bound for $\gamma_t$. Note that we are dividing by $t$ the equation in \cite[Proposition 3.24]{GGZZ26}, so that it has the form we need.
\begin{proposition}
\label{prop:good_distributional_bound}
Let $t>0$ and $\gamma_t, a_t$ and $D_t$ as above. Then
    \begin{equation}
        \label{eq:good_distributional_bound}
        {\Delta \gamma_t}\geq pKa_t-c_t(a_t+\gamma_t)|\de u|^2_{\rm HS}\quad{\rm on}\; B
    \end{equation}
    in the weak sense, for all $t<t^*$ and $p\geqslant2$, where $c_t=1+o_t(1)$ uniformly in $B$. 
\end{proposition}

\section{Main results}
\label{sec:main_res}
We stress again that we shall fix $U\Subset\Omega$ and choose a chain of balls with radii such that the conditions at the end of Section \ref{sec:preliminaries} are in place.
\newline
To prove Theorem \ref{thm:Lipschitz_continuity} we need several preliminary results, starting from the following.

\begin{lemma}
\label{lem:point_Poinc}
Let $u:\Omega\subset\X\to\Y$ be an harmonic map with $u(\Omega)\subset B_\rho(o)$ and $\rho<\pi/2$ and let $U\Subset\Omega$ be open. There exist
$r_0,C>0$ and $\Lambda\geq1$ such that, for every $x,y\in U$ with
$\d_\X(y,x)<r_0$, setting $\ell=\d_\X(y,x)$, we have
\begin{equation*}
\d_\Y\bigl(u(y),u(x)\bigr)
\leq
C\ell
\left(
\fint_{B_{\Lambda \ell}(x)}|\de u|_{\rm HS}^2\de\m
\right)^{1/2}.
\end{equation*}
\end{lemma}

\begin{proof}
We choose $r_0>0$ such that every ball mentioned in the sequel lies
inside $\Omega$. For fixed $a\in U$, the map
\begin{equation*}
z\longmapsto \d_\Y\bigl(u(z),u(a)\bigr)
\end{equation*}
is subharmonic; see, for example \cite[Proposition 3.5]{GGZZ26}.
The Harnack inequality then gives, by \cite[Theorem 3.7]{GGZZ26} and Jensen, applied
at fixed $x$ and with $\ell=\d_\X(x,y)<r_0$,
\begin{equation*}
\d_\Y^2\bigl(u(y),u(x)\bigr)
\leq
C\fint_{B_\ell(y)}
\d_\Y^2\bigl(u(z),u(x)\bigr)\de\m(z).
\end{equation*}

We now do the same with the other variable, so that
\begin{equation*}
\d_\Y^2\bigl(u(y),u(x)\bigr)
\leq
C\fint_{B_\ell(x)}\fint_{B_\ell(y)}
\d_\Y^2\bigl(u(z),u(w)\bigr)\de\m(z)\de\m(w).
\end{equation*}
Since
\begin{equation*}
B_\ell(x)\cup B_\ell(y)\subset B_{2\ell}(x)
\end{equation*}
and the measure $\m$ is locally doubling, we get
\begin{equation*}
\d_\Y^2\bigl(u(y),u(x)\bigr)
\leq
C\fint_{B_{2\ell}(x)}\fint_{B_{2\ell}(x)}
\d_\Y^2\bigl(u(z),u(w)\bigr)\de\m(z)\de\m(w).
\end{equation*}
We now apply the Poincaré inequality (see \cite[Lemma 3.10]{GGZZ26}), to get
\begin{equation*}
\d_\Y^2\bigl(u(y),u(x)\bigr)
\leq
C\ell^2\fint_{B_{\Lambda \ell}(x)}
|\de u|^2_{\rm HS}(z)\de\m(z),
\end{equation*}
proving the lemma.
\end{proof}

We now set
\begin{equation*}
\mathcal{M}(|\de u|_{\rm HS}^2)(x)
:=
\sup_{0<s<\Lambda r_0}
\fint_{B_s(x)}\chi_{2B}|\de u|_{\rm HS}^2\de\m,
\end{equation*}
and shrink $t_0$ so that $B_{\Lambda t_0}\subset 2B$.
The previous lemma then reads as
\begin{equation*}
\d_\Y\bigl(u(y),u(x)\bigr)
\leq
C\,\d_\X(y,x)
\bigl(\mathcal{M}(|\de u|_{\rm HS}^2)(x)\bigr)^{1/2}.
\end{equation*}

We now have another crucial lemma.

\begin{lemma}
There exist $t_0>0$ and an appropriate choice of $p>3$ such that
\begin{equation*}
\sup_{t\in(0,t_0)}
\|\gamma_t\|_{L^2(B)}
<+\infty.
\end{equation*}
\end{lemma}

\begin{proof}
Let $t_0>0$ and $t\in(0,t_0)$, and choose $y_{t,x}\in S_t(x)$ so that
\begin{equation*}
L_t(x)=\d_\X(x,y_{t,x}).
\end{equation*}
By continuity, we can decrease $t_0$ so that
\begin{equation*}
y_{t,x}\in B_{r_0}(x),
\end{equation*}
where $r_0$ is the one of Lemma \ref{lem:point_Poinc}.
We can therefore apply Lemma \ref{lem:point_Poinc}, together with the fact that
$F(s)\leq Cs$, to obtain
\begin{equation*}
D_t(x)
=
F\left(\d_\Y\bigl(u(x),u(y_{t,x})\bigr)\right)
\leq
C L_t(x)
\bigl(\mathcal{M}(|\de u|_{\rm HS}^2)(x)\bigr)^{1/2}.
\end{equation*}
Consequently,
\begin{equation*}
\gamma_t(x)
=
\frac{D_t(x)}{t}-a_t(x)
\leq
C\frac{L_t(x)}{t}
\bigl(\mathcal{M}(|\de u|_{\rm HS}^2)(x)\bigr)^{1/2}
-
\frac{L_t^p(x)}{p\,t^p}.
\end{equation*}
For
\begin{equation*}
A
=
C\bigl(\mathcal{M}(|\de u|_{\rm HS}^2)(x)\bigr)^{1/2},
\end{equation*}
we apply Young's inequality, where
\begin{equation*}
q^{-1}+p^{-1}=1,
\end{equation*}
to get
\begin{equation*}
A\frac{L_t(x)}{t}
-
\frac{L_t^p(x)}{p\,t^p}
\leq
\frac{A^q}{q}.
\end{equation*}
Hence
\begin{equation*}
0\leq\gamma_t
\leq
C\bigl(\mathcal{M}(|\de u|_{\rm HS}^2)(x)\bigr)^{q/2}.
\end{equation*}

Choosing $p$ large enough so that
\begin{equation*}
|\de u|_{\mathrm{HS}}^2\in L^q,
\qquad q>1,
\end{equation*}
which is possible thanks to Theorem \ref{thm:higher_integrability}, we obtain
\begin{equation*}
\int_B\gamma_t^2\de\m
\leq
C\int_B\bigl(\mathcal{M}(|\de u|_{\rm HS}^2)\bigr)^q\de\m
\leq
C\int_{2B}|\de u|_{\rm HS}^{2q}\de\m,
\end{equation*}
which proves the claim. Here we also used the standard maximal-function
inequality.
\end{proof}

The idea is now to integrate $\gamma_t$ in order to obtain important
estimates on its integrated version, which will propagate the regularity.

\begin{lemma}[Time averaging]
Let
\begin{equation*}
0<T<\delta<t_0,
\end{equation*}
with $t_0$ as in the previous lemma, and define
\begin{equation*}
G_T
:=
T\int_T^\delta\frac{\gamma_t}{t^2}\,dt,
\qquad
A_T
:=
T\int_T^\delta\frac{a_t}{t^2}\,dt.
\end{equation*}
Then
\begin{equation*}
G_T,A_T\in W_{\mathrm{loc}}^{1,2}(B)
\end{equation*}
and
\begin{equation*}
(p-1)A_T
=
2G_T+\frac{T}{\delta}\gamma_\delta-\gamma_T
\qquad
\m{\rm -a.e.\,in }\;B.
\end{equation*}
Moreover we have
\begin{equation}
\label{eq:improved_distr_bound}
\Delta G_T
\geq
-pK^-A_T
-c_\delta(A_T+G_T)|\de u|_{\rm HS}^2,
\end{equation}
weakly in $B$, where $c_\delta=\sup_{t\in(0,\delta)}c_t$.
\end{lemma}

\begin{proof}
First let us introduce 
\begin{equation*}
\widehat{L}_t(x)
:=
\max_{y\in S_t(x)}\d_\X(y,x).
\end{equation*}
and then we observe that the map
\begin{equation*}
t\longmapsto f_t(x)
\end{equation*}
is locally Lipschitz on $(0,t_0)$ and
\begin{equation*}
\partial_t^-f_t(x)
=
-\frac{p-1}{p\,t^p}L_t^p(x),
\qquad
\partial_t^+f_t(x)
=
-\frac{p-1}{p\,t^p}\widehat{L}_t^p(x).
\end{equation*}
thanks to the usual differentiability properties of the Hopf-Lax semigroup (see \cite[Theorem 3.3]{ACMcCS21} and the monograph \cite{AmbrosioGigliSavare08} for example).
At every $t$ for which $f_t(x)$ is differentiable, we therefore get
\begin{equation*}
\partial_t f_t(x)
=
-(p-1)a_t(x).
\end{equation*}
Using the relation between $f_t$ and $\gamma_t$, we obtain
\begin{equation*}
(p-1)a_t
=
\gamma_t+t\,\partial_t\gamma_t
\end{equation*}
for $\mathscr{L}^1\otimes m$-a.e. $(t,x)$.

Now fix $T>0$. Up to choosing $\delta>0$ small, there exists a
relatively compact set $W\Subset \widehat{B}$
containing every minimizer $y$ associated with $x\in B$ and
$t\in[T,\delta]$. For these minimizers, set
\begin{equation*}
q_{t,y}(x)
:=
\frac{\d_\X^p(x,y)}{p\,t^{p-1}}
-
F\left(\d_\Y\bigl(u(y),u(x)\bigr)\right).
\end{equation*}
First observe that since $u$ is continuous and $X$ is separable we can actually write $f_t(x)=\inf_{j\in\N}q_{t,y_j}(x)$, where $(y_j)_j$ is a \emph{countable} and dense set of $\overline{\widehat{B}}$. Then we have
\begin{equation*}
|Dq_{t,y_j}|
\leq
C_{T,\delta}\bigl(1+|\de u|_{\mathrm{HS}}\bigr),
\end{equation*}
uniformly in $j\in\N$ and $t\in[T,\delta]$. Since these functions all have the same bound on the weak upper gradient, taking a countable infimum allows that bound to go through the infimization and we get
\begin{equation*}
|Df_t|
\leq
C_{T,\delta}\bigl(1+|\de u|_{\mathrm{HS}}\bigr),
\end{equation*}
and hence
\begin{equation*}
f_t\in W^{1,2}(B).
\end{equation*}
Thanks to this, $G_T$ and $DG_T$ are well-defined as Bochner integrals,
and
\begin{equation*}
DG_T
=
T\int_T^\delta\frac{D\gamma_t}{t^2}\,dt
\in L^2(B).
\end{equation*}
Integrating by parts, we get
\begin{align*}
(p-1)A_T
&=
T\int_T^\delta
\left(
\frac{\gamma_t}{t^2}
+
\frac{\partial_t\gamma_t}{t}
\right)\,dt
\\
&=
2T\int_T^\delta\frac{\gamma_t}{t^2}\,dt
+
\frac{T\gamma_\delta}{\delta}
-
\gamma_T.
\end{align*}
Therefore,
\begin{equation*}
(p-1)A_T
=
2G_T+\frac{T}{\delta}\gamma_\delta-\gamma_T,
\end{equation*}
which also implies that
\begin{equation*}
A_T\in W^{1,2}(B).
\end{equation*}
Finally, to prove \eqref{eq:improved_distr_bound}, let
\begin{equation*}
\varphi\in\operatorname{Lip}_c(B),
\qquad
\varphi\geq0,
\end{equation*}
and integrate \eqref{eq:good_distributional_bound} against the measure
$Tt^{-2}\de t$ to obtain
\begin{equation*}
-\int_B\langle DG_T,D\varphi\rangle\de\m
\geq
-pK\int_BA_T\varphi\de\m
-c_\delta\int_B(A_T+G_T)|\de u|_{\rm HS}^2\varphi\de\m.
\end{equation*}
\end{proof}

\paragraph{Remark.}
Observe that, for $t\in[T,\delta]$, $a_t$ and $\gamma_t$ are uniformly
bounded, so all the terms in the previous lemma are well-defined.

\begin{proposition}
There exist $\delta>0$ and constants $\lambda<2$ and $\mu\geq0$ such
that, for every $T\in(0,\delta)$, we can construct
\begin{equation*}
\overline{G}_T\geq0,
\qquad
\overline{G}_T\in W^{1,2}(B),
\end{equation*}
satisfying
\begin{equation}
\label{eq:improved_bound2}
\Delta\overline{G}_T
\geq
-\mu\overline{G}_T
-\lambda |\de u|_{\rm HS}^2\overline{G}_T
\end{equation}
weakly in $B$. Moreover,
\begin{equation}
\label{eq:G_T_bounded}
\sup_{T\in(0,\delta)}
\left\|\overline{G}_T\right\|_{L^2(B)}
<\infty,
\qquad
\gamma_T\leq2\overline{G}_T
\quad \m{\rm-a.e.}
\end{equation}
\newline
Finally, if
\begin{equation*}
f_{o_B}=\cos\bigl(\d_\Y(u,o_B)\bigr),
\qquad
\psi:=f_{o_B}-\frac12,
\qquad
Z_T:=\frac{\overline{G}_T}{\psi},
\end{equation*}
with $o_B\in B$, then
\begin{equation}
\label{eq:Serb_ineq}
\int_B
\psi^2\left\langle DZ_T,D\varphi\right\rangle\de\m
\leq
\mu\int_B\psi^2Z_T\varphi\de\m
\end{equation}
for every $\varphi\in\operatorname{Lip}_c(B)$ with $\varphi\geq0$.
\end{proposition}

\begin{proof}
We first observe that Proposition \ref{prop:good_distributional_bound} is in place and we choose $p>3$ and $\eta>0$ sufficiently small so that
\begin{equation*}
\lambda
:=
(1+\eta)\left(1+\frac{2}{p-1}\right)
<2.
\end{equation*}
We can then choose $\delta$ so that
\begin{equation*}
c_t\leq1+\eta
\qquad
\text{for }t\in(0,\delta).
\end{equation*}
Set
\begin{equation*}
M_\delta
:=
\|\gamma_\delta\|_{L^\infty(B)}
\leq
\frac{6}{\delta},
\qquad
\overline{G}_T
:=
G_T+\frac{T}{2\delta}M_\delta.
\end{equation*}
Since $A_T\geq0$, the identity for $A_T$ in the previous lemma gives
\begin{equation*}
A_T
\leq
\frac{2}{p-1}\overline{G}_T,
\qquad
\gamma_T
\leq
2\overline{G}_T.
\end{equation*}
Indeed,
\begin{equation*}
(p-1)A_T
\leq
2G_T+\frac{T}{\delta}M_\delta
=
2\overline{G}_T,
\end{equation*}
whereas
\begin{equation*}
\gamma_T
=
2G_T+\frac{T}{\delta}\gamma_\delta-(p-1)A_T
\leq
2\overline{G}_T.
\end{equation*}
Moreover we have
\begin{equation*}
    \|G_T\|_{L^2(B)}\leq T\int_T^\delta\frac{\gamma_t}{t^2}\de t\leq C_0\bigg(1-\frac{T}{\delta}\bigg)\leq C_0,
\end{equation*}
which proves $\sup_{T\in(0,\delta)}
\left\|\overline{G}_T\right\|_{L^2(B)}
<\infty$.
We now apply the previous lemma together with this latter inequality to get
\begin{align*}
\Delta\overline{G}_T
&\geq
-pK^-A_T
-(1+\eta)(A_T+G_T)|\de u|_{\rm HS}^2
\\
&\geq
-\underbrace{\frac{2pK^-}{p-1}}_{\mu}\,
 \overline{G}_T
-
\underbrace{
(1+\eta)\left(1+\frac{2}{p-1}\right)
}_{\lambda}
|\de u|_{\rm HS}^2\overline{G}_T,
\end{align*}
which proves \eqref{eq:improved_bound2}.
\newline
We now remove the energy density from \eqref{eq:improved_bound2} to obtain a better differential inequality.
\newline
Since we assumed
\begin{equation*}
u(\widehat{B})\subset B_s(o_B),
\qquad
s<\frac{\pi}{8},
\end{equation*}
we get
\begin{equation*}
\frac34\leq f_{o_B}\leq1,
\qquad
\frac14\leq\psi\leq\frac12.
\end{equation*}
Moreover,
\begin{equation*}
f_{o_B}\geq2\psi.
\end{equation*}
Now, by Lemma \ref{lem:key_lemma} (here we are also using that $u$ minimizes the energy in every ball among maps with the same boundary values),
\begin{equation*}
\Delta f_{o_B}
\leq
-f_{o_B} |\de u|^2_{\rm HS}.
\end{equation*}
Since $\lambda<2$ and $f_{o_B}\geq2\psi$, this becomes
\begin{equation*}
\Delta\psi
\leq
-\lambda |\de u|^2_{\rm HS}\psi,
\end{equation*}
since we clearly have $\Delta\psi=\Delta f_{o_B}$.

Let now
\begin{equation*}
\varphi\in\operatorname{Lip}_c(B),
\qquad
\varphi\geq0.
\end{equation*}

Testing \eqref{eq:improved_bound2} against $\varphi\psi$ (by approximation) we obtain
\begin{equation}
\label{eq:partial_ineq1}
\int_B
\left\langle
D\overline{G}_T,D(\varphi\psi)
\right\rangle\de\m
\leq
\int_B
\bigl(\mu+\lambda |\de u|_{\rm HS}^2\bigr)
\overline{G}_T\psi\varphi\de\m.
\end{equation}
Now we test the inequality
\begin{equation*}
\Delta\psi
\leq
-\lambda |\de u|_{\rm HS}^2\psi
\end{equation*}
against $\overline{G}_T\varphi$. We obtain
\begin{equation}
\label{partial_ineq2}
\int_B
\left\langle
D\psi,D(\overline{G}_T\varphi)
\right\rangle\de\m
\geq
\lambda\int_B
|\de u|_{\rm HS}^2\psi\overline{G}_T\varphi\de\m.
\end{equation}
Subtracting \eqref{partial_ineq2} from \eqref{eq:partial_ineq1} gives
\begin{equation*}
\int_B
\left\langle
\psi D\overline{G}_T-\overline{G}_T D\psi,
D\varphi
\right\rangle\de\m
\leq
\mu\int_B
\psi\overline{G}_T\varphi\de\m.
\end{equation*}
Since $\psi\geq\frac14$, setting $Z_T:=\overline{G}_T/\psi\in W^{1,2}_{\rm loc}(B)$, the chain rule gives
\begin{equation*}
\psi D\overline{G}_T-\overline{G}_T D\psi
=
\psi^2D\left(\frac{\overline{G}_T}{\psi}\right)
=
\psi^2DZ_T,
\end{equation*}
therefore proving \eqref{eq:Serb_ineq}.
\end{proof}

\begin{lemma}
Let
\begin{equation*}
Z\geq0\quad{\rm with}\;Z\in W^{1,2}_{\rm loc}(B)
\end{equation*}
and suppose that, for every
$\varphi\in\operatorname{Lip}_c(B)$ with $\varphi\geq0$,
\begin{equation}
\label{eq:weak_Moser}
\int_{B}
\theta\langle DZ,D\varphi\rangle\de\m
\leq
\mu\int_{B}
\theta Z\varphi\de\m,
\end{equation}
for some $\mu\geq 0$.
Assume that $\theta\in L^\infty(B)$, with
\begin{equation*}
0<\theta_0\leq\theta\leq\theta_1<+\infty\quad\m-{\rm a.e.\,in}\;B.
\end{equation*}
Then for all $B_R(x_0)$ such that $B_{2R}(x_0)\Subset B$ we have
\begin{equation}
\label{eq:giga_Harnack}
\operatorname*{ess\,sup}_{B_R(x_0)}Z
\leq
C
\left(
\fint_{B_{2R}(x_0)}Z^2\de\m
\right)^{1/2},
\end{equation}
where
\begin{equation*}
C=C(\theta_0,\theta_1,K,R,|\mu|)>0
\end{equation*}
is equibounded as $R\to 0$.
\end{lemma}

\begin{proof}
Let $\beta\geq2$ and let $\eta$ be a Lipschitz cutoff function.
Using $\eta^2Z^{\beta-1}$ as a test function in \eqref{eq:weak_Moser} (via approximation), we obtain
\begin{equation*}
\int_{B_{2R}(x_0)}
\eta^2
\left|D\left(Z^{\beta/2}\right)\right|^2\de\m
\leq
C\frac{\theta_1}{\theta_0}\beta^2
\int_{B_{2R}(x_0)}
\bigl(|D\eta|^2+\mu\eta^2\bigr)Z^\beta\de\m.
\end{equation*}
Indeed, using the chain rule, we have
\begin{align*}
&
\int_{B_{2R}(x_0)}
\eta^2\theta
\left\langle
DZ,D\left(Z^{\beta-1}\right)
\right\rangle\de\m
\\
&\quad+
\int_{B_{2R}(x_0)}
\theta Z^{\beta-1}
\left\langle DZ,D(\eta^2)\right\rangle\de\m
\\
&\leq
\mu\int_{B_{2R}(x_0)}
\theta Z^\beta\eta^2\de\m.
\end{align*}

Now
\begin{equation*}
D\left(\eta^2Z^{\beta-1}\right)
=
2\eta Z^{\beta-1}D\eta
+
(\beta-1)\eta^2Z^{\beta-2}DZ,
\end{equation*}
so that
\begin{align*}
(\beta-1)
\int_{B_{2R}(x_0)}
\theta\eta^2Z^{\beta-2}|DZ|^2\de\m
&\leq
2\int_{B_{2R}(x_0)}
\theta\eta Z^{\beta-1}|DZ|\,|D\eta|\de\m
\\
&\quad+
\mu\int_{B_{2R}(x_0)}
\theta\eta^2Z^\beta\de\m.
\end{align*}
Using Young's inequality and absorbing terms gives
\begin{align*}
\int_{B_{2R}(x_0)}
\theta\eta^2Z^{\beta-2}|DZ|^2\de\m
&\leq
\frac{C}{(\beta-1)^2}
\int_{B_{2R}(x_0)}
\theta Z^\beta|D\eta|^2\de\m
\\
&\quad+
\frac{C\mu}{\beta-1}
\int_{B_{2R}(x_0)}
\theta\eta^2Z^\beta\de\m.
\end{align*}
Since
\begin{equation*}
\left|D\left(Z^{\beta/2}\right)\right|^2
=
\frac{\beta^2}{4}
Z^{\beta-2}|DZ|^2,
\end{equation*}
we conclude with the previous estimate, using also
\begin{equation*}
\theta_0\leq\theta\leq\theta_1.
\end{equation*}
Finally \eqref{eq:giga_Harnack} follows from the local Sobolev inequality on ${\rm RCD}(K,N)$ spaces, along with the classical De Giorgi-Nash-Moser iteration.
\end{proof}

\begin{proposition}
\label{prop:unif_bounds_HL}
For every $V\Subset B$, there exist $\delta>0$ and $C_V>0$ such that
\begin{equation*}
\sup_{0<T<\delta}
\|\gamma_T\|_{L^\infty(V)}
\leq C_V.
\end{equation*}
\end{proposition}

\begin{proof}
Cover $V$ with finitely many balls $B_R(x_i)$ such that
\begin{equation*}
B_{2R}(x_i)\Subset B.
\end{equation*}
Apply the previous lemma to $Z_T$ with $\theta=\psi^2$, so that
\begin{equation*}
\theta_0
:=
\frac{1}{16}
\leq
\psi^2
\leq
\frac14
=:
\theta_1.
\end{equation*}
Now
\begin{equation*}
Z_T^2
=
\frac{\overline{G}_T^{\,2}}{\psi^2}
\leq
16\overline{G}_T^{\,2}.
\end{equation*}
The $L^2$-norms of $\overline{G}_T$ are uniformly bounded thanks to \eqref{eq:G_T_bounded}, i.e. $\sup_{T\in(0,\delta)}
\left\|\overline{G}_T\right\|_{L^2(B)}
<\infty$. We can then apply $\eqref{eq:giga_Harnack}$ and infer
\begin{equation*}
Z_T\in L^\infty\bigl(B_R(x_i)\bigr)
\end{equation*}
with a bound independent of $T$ and depending on $\|Z_T\|_{L^2(B)}$. Since
\begin{equation*}
\overline{G}_T=\psi Z_T
\qquad\text{and}\qquad
\gamma_T\leq2\overline{G}_T,
\end{equation*}
we obtain the claim, exploting also the H\"older continuity of $u$ (which allows to use a supremum instead of an essential supremum).
\end{proof}
We are finally ready to prove the main Theorem.
\begin{proof}[Proof of Theorem \ref{thm:Lipschitz_continuity}]
Fix $\hat{x}\in V\Subset B$ 
satisfying the nested ball assumptions introduced at the end of Section \ref{sec:preliminaries} together with $u(\widehat{B})\subset B_s(u(\hat{x}))$, with $s<\pi/8$.
Then, by Proposition \ref{prop:unif_bounds_HL},
\begin{equation*}
\gamma_t(x)\leq C_V
\qquad
\text{for every }x\in V,\quad t\in(0,\delta),
\end{equation*}
for some $\delta>0$. Let
\begin{equation*}
r=\d_\X(x,y),
\qquad
x,y\in V.
\end{equation*}
Assume first that $r\in(0,\delta)$. Then
\begin{align*}
f_r(x)
&\leq
\frac{r^p}{p\,r^{p-1}}
-
F\left(\d_\Y\bigl(u(y),u(x)\bigr)\right)
\\
&=
\frac{r}{p}
-
F\left(\d_\Y\bigl(u(y),u(x)\bigr)\right).
\end{align*}
Since
\begin{equation*}
f_r=-r\gamma_r,
\end{equation*}
it follows that
\begin{equation*}
c\,\d_\Y\bigl(u(y),u(x)\bigr)
\leq
F\left(\d_\Y\bigl(u(y),u(x)\bigr)\right)
\leq
r\left(\frac1p+\gamma_r(x)\right)
\leq
\left(\frac1p+C_V\right)r.
\end{equation*}
If instead $r\geq\delta$, then
$u(V)\subset B_s\bigl(u(\hat{x})\bigr)$ gives
\begin{equation*}
\d_\Y\bigl(u(y),u(x)\bigr)
\leq
2s
\leq
\frac{2s}{\delta}\,\d_\X(x,y).
\end{equation*}
Since $x_0$ was arbitrary, $u$ is locally Lipschitz in $\Omega$.
\end{proof}

\noindent \textbf{Acknowledgements.} 
This work was supported by UK Research and Innovation (UKRI) under the Horizon Europe funding guarantee [grant number EP/Z000297/1]. The author wishes to thank Nicola Gigli for helpful comments on a preliminary draft of this work.
\bibliography{references}
            \bibliographystyle{alpha}

\end{document}